\documentclass[letterpaper, 10 pt, conference]{ieeeconf}  

\IEEEoverridecommandlockouts                              
\usepackage[utf8]{inputenc}
\usepackage[english]{babel}
\usepackage{amsmath}
\usepackage{amssymb}
\usepackage{epsfig}
\usepackage{float}
\usepackage{graphicx}
\usepackage{hyperref}
\usepackage{color}
\usepackage{framed}
\usepackage{subfigure}
\usepackage{bm}
\usepackage{eucal}
\usepackage{mathtools}
\usepackage{thmtools}
\usepackage{caption}
\usepackage{comment}
\usepackage{todonotes}

\newtheorem{thm}{Theorem}

\newtheorem{cor}{Corollary}

\newcommand{\Ev}[1]{\mathcal{E}_{#1}}

\newcommand{\Pro}[1]{P_{\Omega_{#1}}}

\newcommand{\IO}[1]{I_{\Omega_{#1}}}

\newcommand{\RR}{{\mathbb{R}}}
\newcommand{\NN}{{\mathbb{N}}}
\newcommand{\ZZ}{{\mathbb{Z}}}
\newcommand{\Ker}{{\mathfrak{K}}}

\newcommand{\dkoop}{{\mathfrak{U}}}
\newcommand{\koop}{U}

\title{\LARGE \bf
Intrinsic and Extrinsic  Approximation  of \\ Koopman Operators over Manifolds
}

\author{Sai Tej  Paruchuri$^{1,*}$, Jia Guo$^{2,*}$, Michael Kepler$^{3,\dagger}$, Tim Ryan$^{*}$, Haoran Wang$^{*}$,\\ Andrew J. Kurdila$^{4,*}$, and  Daniel Stilwell$^{5,\dagger}$ 
\noindent \thanks{\noindent \newline$^1$PhD candidate, {\tt \small saitep@vt.edu},  \newline 
    $^2$PhD candidate, {\tt \small jguo18@vt.edu}, \newline 
    $^3$PhD candidate, {\tt \small mkepler@vt.edu}, \newline 
    $^4$W. Martin  Johnson Professor, {\tt\small kurdila@vt.edu}, \newline
    $^5$Professor, {\tt\small stilwell@vt.edu}, \newline
    $^{*}$Mechanical Engineering Dept., Virginia Tech, Blacksburg, VA \newline
    $^{\dagger}$Electrical and Computer Engineering Dept., Virginia Tech, Blacksburg, VA \newline
        }%
}

\begin{document}

\maketitle
\thispagestyle{empty}
\pagestyle{empty}

\begin{abstract}

This paper derives rates of convergence of certain approximations of the Koopman operators that are associated with discrete, deterministic, continuous semiflows on a complete metric space $(X,d_X)$. Approximations are constructed in terms of reproducing kernel bases that are centered at samples taken along the system trajectory.   It is proven that when the samples are dense in a certain type of smooth manifold $M\subseteq X$, the derived rates of convergence depend on the fill distance of samples along the trajectory in that manifold. Error bounds for projection-based and data-dependent approximations of the Koopman operator are derived in the paper. A discussion of how these bounds are realized in intrinsic and extrinsic approximation methods is given. Finally, a numerical example that illustrates qualitatively the convergence guarantees derived in the paper is given.  

\end{abstract}

\section{INTRODUCTION}
\label{sec:introduction}
Over the past decade an extensive literature has been archived on Koopman theory,  and more generally on data-dependent approaches, for modeling various types of nonlinear systems. An idea of the breadth of applications of the theory can be gained by considering the work in   \cite{klus2016,ss2013} for studies of molecular dynamics,  or \cite{sra2019, zrdc2019,kh2018, gkks2018,hrdc2017} for applications to the study of fluid flows, or \cite{tlld2015, tbd2015} in the atmospheric sciences.  A good account of the basics  underlying Koopman theory can be found in texts like \cite{lasota} or \cite{nagelergodic}. Recent notable references that study the general  methodology of Koopman theory, with an emphasis on topics related to approximation theory,  include  \cite{gdosjzz2019, g2019, gkes2019, rabk2019,du2019,cbk2019,lv2018,klus2018,klus2020,pd2018,mcm2018,dmm2018,bldk2018,hnfdl2017,ak2017,kbbp2016a, pbk2016,wkc2015,pw2015,korda2018,kurdila2018,dgs2019,dasgiannakis2019,ag2019}. All of these latter papers have appeared over the past five years. 

The motivation for employing Koopman methods is now well-known: the theory provides an approach to the study of uncertain systems that makes extensive use of operator theory to enhance the understanding of the unknown dynamics. The theory is generally applicable to, indeed in a sense expressly designed for, the study of nonlinear systems. Koopman theory provides an elegant framework in which to carry out analysis of uncertain nonlinear dynamics as well as to develop  data-driven algorithms for modeling and identification of such systems.  To be sure, there are both theoretical and pragmatic reasons for the popularity of Koopman methods.

As explained well in a number of other references such as \cite{koopmanism},  and in greater detail than is possible in this short conference paper, there is a fundamental trade-off in applying Koopman theory to a given nonlinear system.  If we have  a nonlinear system whose dynamics is poorly understood, Koopman theory {\em in principle} entails replacing the study of the system of interest, which is nonlinear and finite dimensional, with one that is linear and infinite dimensional. Since practical considerations   dictate that finite dimensional representations are needed, questions regarding the convergence of approximations must be addressed in any full understanding of Koopman theory. 
 
Unfortunately, many of the finer points regarding the convergence of approximations of the Koopman operator are necessarily nuanced.
The  large number of explicitly cited  papers above  that have appeared over the past five years or so have carefully studied various questions related to convergence  of Koopman approximations. 
In general, these studies  build approximations of quantities or mathematical  objects associated with the unknown flow from samples or  observations. The approximations can take the form of   estimates or predictors  of the state, estimates  of an observable function, or approximations of  the propagation law of the dynamical system itself, among other examples.  These references study a diverse number of cases and provide numerous precise sufficient conditions that guarantee that  convergence is achieved as the dimension $n$  of the space of approximants $n$ and/or the number $m$ of samples  approach infinity.

As motivation for this paper, it is useful to compare this  state of the art in Koopman theory to that in the field of evolutionary partial differential equations (PDEs) or nonlinear regression. Several decades of research in these  fields has resulted in a rich theory that relates rates of convergence  of approximations  to the choice of bases. Here, when we refer to rates of convergence we mean error bounds that are explicit in the dimension $n$ of the space of approximants or the number of samples $m$, or both.  As noted  above, sufficient conditions that ensure convergence asymptotically as $n
 \rightarrow \infty$ or $m\rightarrow \infty$  are numerous in Koopman theory, whereas rates of convergence are far less common.  There are many reasons for this. Approximations of the Koopman operator are typically generated using samples along the trajectory of an uncertain dynamical system, and consequently the domain over which approximations are to  be constructed can be unknown {\em a priori}. This means that much of the ``standard machinery'' that is brought to bear in the numerical study of evolutionary PDEs over a known domain 
 -- approximations in piecewise polynomial, finite element, spline, or wavelet spaces --  can be problematic in applications to Koopman operators. Moreover, it is often of primary concern that approximations of the Koopman operator can be used subsequently to generate  approximate models of  dynamics that are somehow consistent with the underlying unknown dynamics. It seems that questions of rates of convergence are of secondary concern perhaps in these applications where it is primarily  desired to obtain approximate dynamics that preserve some structure in the underlying unknown dynamics.

\subsection{Summary of New Results} 
In this paper, a number of new results are derived that make precise the rates of convergence of  approximations of some types of Koopman operators that are associated with deterministic flows on manifolds. 
 \subsubsection{The Problem Setup and  Formulation} We begin the analysis in this paper by assuming that we have  discrete deterministic semiflow on a state space that is a  complete metric space $(X,d_X)$.   Continuity of the semiflow is defined in terms of the metric $d_X$ on the state space $X$. The continuous semiflow is induced by the autonomous recursion 
 \begin{equation}
 \phi_{n+1}=f(\phi_n)
     \label{eq:deterministic}
 \end{equation}
 for some unknown function $f:X\rightarrow X$. 
 Approximation results derived in this paper  are stated for the Koopman operator $\koop_f g:=g\circ f$.  
 We let $\Omega_n:=\{\xi_i \in X \ | \ 1\leq i \leq n\}$ denote a finite set of observations of the state of the system, and the complete set of samples associated with some fixed initial condition 
 is denoted $\Xi:=\Xi(\phi_0):=\bigcup_{n\in \NN} \Omega_n \subset \Gamma^+(\phi_0)$. Here   $\Gamma^{+}(\phi_0):=\bigcup_{i\in \ZZ^+}\phi_i$ is the forward orbit through $\phi_0\in X$. The samples $\Xi$ are assumed to be dense in a limiting set $\Omega$, which may coincide with the entire state space $\Omega=X$,  or it can be a proper subset  $\Omega\subset X$.  One of the essential features of this paper is that the rates of convergence of approximations of the Koopman operator, which apply when it so happens that the limiting set $\Omega$ is a smooth manifold $M$, are given in  terms of the fill distance $h_{\Omega_n,\Omega}$ of the finite collection of samples $\Omega_n$ in the limiting set $\Omega$,
\begin{equation}
    \label{eq:filldistance}
    h_{\Omega_n,\Omega}:=\sup_{x\in \Omega} \min_{\xi_i\in \Omega_n} d_X(x,\xi_i).
\end{equation}
Note that since we want $h_{\Omega_n,\Omega}\rightarrow 0$ as $n\rightarrow \infty$, it must be the case that the limiting set $\Omega$ is bounded in the analysis in this paper.

 Realizations  of approximations to the Koopman operator are built in this paper  using finite dimensional spaces of approximants $H_{\Omega_n}:=\text{span}\{ \Ker_{X,\xi_i} \ | \ \xi_i \in \Omega_n \}$ with $\Ker_{X,\xi_i}(\cdot):=\Ker_X(\xi_i,\cdot)$ the basis function centered at the sample $\xi_i\in X$, where $\Ker_X:X\times X \rightarrow \RR$ is the kernel function that induces the native space of the reproducing kernel Hilbert space $H_X$.  
 \subsubsection{Projection-Based Approximations}
 The first new result of the paper is stated in  Theorem \ref{th:many_M}, and it applies when $\Omega=M$ is in fact a smooth, compact, connected, Riemannian manifold. In this case we  select the native space $H_M$ so that it is continuously embedded in a Sobolev space $W^{t,2}(M)$  of high enough order.   This theorem gives sufficient conditions to ensure that the projection-based Koopman operator $\koop_f^n:=(P_{\Omega_n}(\cdot))\circ f$ satisfies a bound that has the form  
\begin{equation}
\label{eq:summary1}
\|\koop_f g - \koop_f^n g\|_{f^*(W^{s,2}(M))} \lesssim h^{t-s}_{\Omega_n,M} \|g\|_{W^{t,2}(M)}
\end{equation}
for all $g$ in the Sobolev space $  W^{t,2}(M)$, provided that the  limiting set $\Omega$ is in fact a smooth, connected, compact, Riemannian manifold $\Omega:=M$. In this equation the error is measured in the pullback space $f^*(W^{s,2}(M))$, defined in Section \ref{sec:RKHmain}. The ranges for the smoothness  indices $t,s$ are dictated by the Sobolev embedding theorem and the ``many zeros'' theorem (Theorem \ref{th:many_zeros_M}) on manifolds.  The bound in Equation \ref{eq:summary1} as of yet has no analog in the series of recent articles cited above for approximations of the Koopman operator.

Bounds on the error induced by the projection-based approximation $\koop_f^n$ are certainly valuable to understand the ``worst-case'' performance of approximations built from a given finite dimensional space of approximants $H_{\Omega_n}$, and they are also important in their role in studying {\em data-dependent} approximations $\dkoop_f^n g:=P_{\Omega_n}\left ( (P_{\Omega_n} g )\circ f \right )$, discussed next. 

\subsubsection{Data-Dependent Approximations} 
The initial  approximation $\koop_f^n(\cdot):=(P_{\Omega_n}(\cdot))\circ f$ in Equation \ref{eq:summary1}  uses the projection operator $P_{\Omega_n}:H_M\rightarrow H_{\Omega_n}$, but this expression cannot be evaluated unless the function $f$ is known.   As shown  in Section \ref{sec:koop_approx}, the operators $\dkoop_f^n$ can be constructed from the input-output samples $\{(\phi_i,y_i)\}_{1\leq i\leq  n}=\{(\phi_i,f(\phi_i)\}_{1\leq i \leq  n}$ along the discrete trajectory of the system in Equation \ref{eq:deterministic}. It is also worth noting that the realization of the coordinate representation of $\dkoop_f^n$ is closely related to the approximation of the Koopman operator that is defined in terms of the Extended Dynamic Mode Decomposition (EDMD) algorithm \cite{wkc2015},  in the special case  that  the number of samples is equal to the dimension of the space of approximants. 
The  definition above of $\dkoop_f^n$  makes sense only so long as $(P_{\Omega_n}g)\circ f \subseteq H_M$. Thus, a standing assumption in this case is that the pullback space $f^*(H_M)\subseteq H_M$. Since $(P_{\Omega_n}g)\circ f \in f^*(H_M)$, this structural assumption is enough to ensure that the data-driven operator $\dkoop_f^n$ is well-defined. Theorem \ref{th:many_M_data} is  representative of the type of  bound that can derived in this case. We have a pointwise error bound 
\begin{align*}
|(\koop_f g)(x)&-(\dkoop_f^n g)(x)|  \\
  \hspace*{.5in}  &\leq C_M h^{t-s}_{\Omega_n,M} \biggl  (\|g\|_{W^{t,2}(M)}+ 
 \|\koop_fg\|_{W^{t,2}(M)}\biggr )
\end{align*}
for each $x\in M$ in terms of the fill distance of the samples $\Omega_n$ in the manifold $M$.
Again, this result is novel among  the articles in the recent literature on approximation of Koopman operators. 

\subsubsection{Intrinsic and Extrinsic Approximations}
The representations of the approximations of the Koopman operator in this paper are explicit in terms of the kernel basis $\Ker_{X,\xi_i}$ that is defined over the state space $X$, where $\Ker_X:X\times X\rightarrow \RR$ is the kernel that defines the RKH space $H_X$.  In cases when the samples $\Xi$ are dense in $X$, the kernel basis $\Ker_{X,\xi_i}$ is defined from a kernel defined on all of $X$. A critical feature of the error bounds in the paper is that they are derived by assuming that the kernel induces a native space $H_X$ that is embedded in or equivalent to a Sobolev space. In particular applications coming up with the needed closed form expressions for a kernel can sometimes be difficult. For this reason, we describe both intrinsic and extrinsic realizations of the approximation framework in this paper, which we describe next. 

In all of the theorems developed in this paper, the limiting set $\Omega$ is assumed to be a smooth Riemannian manifold $M:=\Omega$. In some cases the limiting set fills the entire state space $X=M=\Omega$, and in others it is a proper subset $M=\Omega\subset X$.  When the limiting set $\Omega=M$ is in fact the entire state space $X$, it is possible to use an intrinsic approximation method since the manifold is known in this case.  
When we say that an approximation method is intrinsic, we mean that the kernel used in approximations is defined in terms of the intrinsic definition of the manifold $M$. For example, the kernel may be defined in terms of the eigenfunctions of a  differential operator on the manifold.  The approximant spaces in this case require a closed form expression for the kernel, which in turn   requires a closed form expression for the eigenfunctions on the manifold.
Overall, a fine analysis of rates of convergence for intrinsic approximations of functions are  described in the set of papers \cite{HNW,HNRW,HNWpoly}. 
However, despite the attractiveness {\em in principle} of using such an intrinsic method here, such an approach is difficult in building approximations of Koopman operators. Coming up with the required  closed form expressions is a nontrivial task for a general Riemannian 
manifold $M$ and  requires detailed  knowledge of the form of the manifold $M$. Section \ref{sec:examples} examines one case that illustrates the challenges in devising intrinsic approximations, even in the case of simple recursions over a manifold.

However, it is perhaps most usually the case in practical problems that the samples $\Xi$ do not fill the entire state space $X$.  Rather, the limiting set $\Omega$ in which the samples $\Xi$ are dense is typically not known. In this case, even if the limiting set $\Omega$ is a nice smooth manifold, it is impossible to use a kernel basis $\Ker_{M,\xi_i}$ that is defined intrinsically with respect to the  manifold $M:=\Omega$. In this latter case we employ an {\em extrinsic} approximation. A general study of extrinsic methods for approximation of functions can be found in \cite{fuselier}. We choose a kernel  $\Ker_X$ that is well-defined and known on the large state space $X$, and we define a kernel  on the manifold $M$ by restriction.  Even though the manifold $M$ is not known, if we are given samples that reside on $M$, all the coordinate realizations of the approximations of the Koopman operator can still be computed. Moreover, the  rates of approximation above can still be shown to hold when restricted to a regularly embedded  submanifold $M\subset X$. Since the submanifold is a set of zero measure as a subset of $X$, there is some loss of regularity that reduces the guaranteed rate of convergence.  We outline this analysis in Section \ref{sec:examples}.

\section{CONSTRUCTIONS IN RKH SPACES}
\label{sec:RKHmain}
As mentioned in Section \ref{sec:introduction}, $H_X$ is an RKH space of real-valued functions over $X$. In this section, we review relevant definitions and properties of the RKH space $H_X$, the restricted RKH space of functions over the manifold $M \subset X$ $H_M$, and the pullback space $f^*(H_M)$ where $f: M \to M$. This section also includes a brief discussion of the interpolation and the projection operators defined on RKH spaces. 

\subsection{RKH Space $H_X$ and $H_M$ of Functions}
\label{ss: RKH Space}
A symmetric, continuous, real valued function $\mathfrak{K}_X:X\times X \to \mathbb{R}$, is a reproducing kernel if it is a positive type function, i.e. for any finite collection of points $\{\xi_i\}_{1\leq i\leq n}\subseteq X$, the Grammian $\mathbb{K}_{X,n}:= [\mathfrak{K}_X (\xi_i,\xi_j)]$ is a positive semi-definite matrix. All such positive type functions induce an reproducing kernel Hilbert (RKH) space $H_X$ that is defined as
$
H_X := \overline{\text{span}\{\Ker_{X,x} \ | \ x\in X \}},
$
where $\Ker_{X,x}(\cdot)$ is the kernel centered at $x \in X$ and is equal to $\Ker_X(x,\cdot)$. The inner product $(\cdot,\cdot)_{H_X}$ of the Hilbert space $H_X$ is defined as $(\Ker_{X,x}, \Ker_{X,y})_{H_X} := \Ker_X(x,y)$ for any two functions $\Ker_{X,x},\Ker_{X,y} \in H_X$ and for all $x,y\in X$. It satisfies the reproducing property $(f,\Ker_{X,x})_{H_X}=f(x)$ for all $f\in H_X$ and $x\in X$. Not all Hilbert spaces are RKH spaces. A necessary and sufficient condition for a Hilbert space to be an RKH space is the boundedness of the evaluation functional $\mathcal{E}_x:f \to f(x)$ for any $x \in X$. In our analysis, we assume that the evaluation functional is in fact uniformly bounded, i.e. there exists a constant $\bar{k}$ such that $\| \mathcal{E}_x \| \leq \bar{k}$ for all $x \in X$. This assumption guarantees that the RKH space is embedded into the space of continuous function $C(X)$, that is, $H_X\hookrightarrow C(X)$. If the manifold $M = X$, the RKH space $H_M = H_X$. However, when $M \subset X$ and the intrinsic structure of $M$ is not exactly known, we define the space $H_M$ by restricting the kernel $\Ker_X$ to $M \times M$. The restriction of $\Ker_{X}$, $\Ker_M:M \times M \to \RR$, is defined as $\Ker_M(x,y):= \Ker_X |_{M \times M} (x,y)$ for all $x,y \in M$. Naturally, we can define the space $H_M$ using the kernel $\Ker_M$ similar to the way we defined $H_X$. The space $H_M$ is itself an RKH space and its inner product is defined in terms of the kernel $\Ker_M$. Alternatively, if $R_M$ represents the restriction operator to $M$, we can define $H_M$ as $H_M = R_M (H_X) := \{ R_M f | f \in H_X \}$. As mentioned in Section \ref{sec:introduction}, spaces of the form $H_M$ are particularly useful when the samples $\Xi$ of the dynamical system are concentrated in $M$ and not the whole space $X$.


\subsection{The Pullback RKH Spaces $\gamma^*(H_M)$ for $\gamma:S\rightarrow M$}
\label{ss_pullbackRKHS}
The pullback space $\gamma^*(H_M)$ generated by  the  space of functions $H_M$ and any mapping $\gamma : S \to \mathbb{R}$  is defined to be 
\begin{align}
\gamma^*(H_M) := \left \{ g:S \to \mathbb{R}\  \biggl | \ g = h \circ \gamma, h \in H_M \right \}
\end{align}
for any set $S$. 
By definition, the Koopman operator $U_f$ maps an element of $H_M$ to its pullback space $f^*(H_M)$. When $H_M$ is a general normed vector space with the norm $\| \cdot \|_{H_M}$, the norm of the pullback space is defined as 
\begin{align}
\|g\|_{\gamma^*(H_X)}:=\min \left \{ \|h\|_{H_M} \ \biggl | \ 
g=h\circ \gamma, \ \ h\in H_M
\right \}.
\end{align}
When $H_M$ is an RKH space, which is what we assume in this paper, the pullback space $\gamma^*(H_M)$ is itself an RKH space with the kernel $\Ker_{M,\gamma}$ defined as
\begin{align}
\Ker_{M,\gamma}(\tau,s):=\Ker_M(\gamma(\tau),\gamma(s))
\end{align}
for all $\tau,s \in S$. In other words, the kernel $\Ker_{M,\gamma}$ generates the pullback space $\gamma^*(H_M)$, i.e. $\gamma^*(H_M):=\overline{\text{span}\{\Ker_{M,\gamma,s}\ | \ s\in S\}}$ with $\Ker_{M,\gamma,s}:=\Ker_M(\gamma(s),\gamma(\cdot))$ for each $s\in S$. 

\subsection{Interpolation and Projection}
\label{sec:interp}
The space $H_M$ discussed in the previous subsection is infinite-dimensional and the Koopman operator $U_f$ maps this space to corresponding infinite-dimensional dimensional pullback space $f^*(H_M)$. We define the approximation of the Koopman operator in terms of a certain  finite-dimensional subspace of $H_M$. Let $\Omega_n :=\{\xi_1,\ldots, \xi_n\} \subseteq M$ be a set of $n$ points, and let $H_{\Omega_n} := \text{span}\{\Ker_{M,\xi_i}\ | \ \xi_i \in \Omega_n\}$ be the corresponding RKH space. We define the orthogonal projection operator $P_{\Omega_n} : H_M \to H_{\Omega_n}$ as the unique mapping that satisfies the identity
\begin{align}
((I-P_{\Omega_n})h,g)_{H_M}=0
\end{align}
for all $g\in H_{\Omega_n}$ and $h\in H_M$. The projection operator decomposes the space $H_M$ into $H_M = H_{\Omega_n} \oplus V_{\Omega_n} $, where $V_{\Omega_n} :=\{ f\in H_M \ | \ f|_{\Omega_n}=0 \}$. We define   the interpolation operator $\IO{n}: H_M \rightarrow H_{\Omega_n}$ to be the unique operator that satisfies the interpolation conditions  
$$
(\IO{n} f)(\xi_i)=f(\xi_i)
$$
for all $\xi_i\in \Omega_n$ and $f\in H_M$. For RKH spaces, the interpolation operator is identical  to the projection operator, in other words,  $\IO{n}f=P_{\Omega_n}f$ for all $f \in H_M$.

\subsection{Sobolev Spaces over Riemannian Manifolds $M$}
\label{sec:sobolev_M}
Suppose we have a (smooth) Riemannian manifold $M$ with metric $g_p$ and inner product $(\cdot, \cdot)_{g_p}$ on the tangent space $T_pM$ at point $p \in M$. When $r$ is an integer, the Sobolev space $W^{r,2}(\Omega)$ for a subset $\Omega \subseteq M$ contains all the functions in $L^2(\Omega)$ such that the norm induced by the inner product
\begin{equation}
(f,g)_{W^{k,2}(\Omega)}:=
\sum_{0\leq j \leq r}\int_\Omega
(\nabla^j f, \nabla^jg)_{g,p}d\mu(p)
\label{eq:sob_man}
\end{equation}
is bounded. In the above definition of the inner product, the term $\mu$ is the volume measure on the manifold $M$. Given a set of coordinates $(x^1,\ldots,x^d)$, the volume measure is defined as $d\mu(x):= \sqrt{det(g)} dx^1\ldots dx^d$. For real-valued $r>0$, the Sobolev space $W^{r,2}(\Omega)$ is defined as an interpolation space between the integer order Sobolev space and $L^2(\Omega)$.  A central theorem we use to prove the results of this paper is a simplified version of the ``many zeros'' theorem \cite{wendland,HNW,HNRW} given below. 
\begin{thm}
\label{th:many_zeros_M}
Suppose that $M$ is a smooth $d$-dimensional manifold. Let $t\in \mathbb{R}$ with  $t>d/2$, $s\in \mathbb{N}_0$ with  $0\leq s \leq \lceil t\rceil -1$.
Then there are   constants $h_M,C_M>0$ such that for all $\Omega_n\subset \Omega$ such that the fill distance $h_{\Omega_n,M}\leq h_M$ and  for all $u\in W^{t,2}(M)$ that  satisfies $u|_{{\Omega_n}}=0$, we have 
$$
\|u\|_{W^{s,2}(M)}\leq C_M h_{\Omega_n,M}^{t-s}\|u\|_{W^{t,2}(M)}.
$$
\end{thm}

\subsection{Relationships between RKH Spaces and Sobolev Spaces} 
\label{ss_RKHSSobolev}
In this paper, we derive the convergence results and approximation rates when the RKH space $H_M$ is embedded in a Sobolev space $W^{r,2}(M)$ for real $r>0$.
When the manifold $M$ is a $d$-dimensional, connected, smooth, Riemannian manifold having a positive radius of injectivity and bounded geometry, by the Sobolev embedding theorem, we have $W^{r,2}(M) \overset{i}{\hookrightarrow} C(M)$ for $r > d/2$. When this is true, we have
\begin{align*}
    |\Ev{x}f|=|f(x)| \leq \|f\|_{C(M)}\leq C \|f\|_{W^{r,2}(M)}.
\end{align*}
This shows that the evaluation functional is bounded, which in turn implies that $W^{r,2}(M)$ is a RKH space when $r> d/2$. A discussion of these results can be found in \cite{HNW,HNRW,HNWpoly}. 

%
\section{APPROXIMATIONS OF THE KOOPMAN OPERATOR $\koop_f$}
\label{sec:koop_approx}
This section presents the principal results of this paper. We present error rates for two different types of approximations of the Koopman operator $U_f$, (i) the projection-based approximation  $U_f^n: = U_f P_{\Omega_n}$, and  (ii) the data-dependent approximation $\dkoop_{f}^n := P_{\Omega_n} ( (P_{\Omega_n} g) \circ f)$.

We define the first approximation of the Koopman operator $U_f^n$ as
\begin{align}
\label{eq:Ufn1}
\koop_{f}^ng=(P_{\Omega_n} g)\circ f,
\end{align}
where $f:M \to M$. When the samples are dense in the manifold $M$, we can express this finite-dimensional approximation using the relation
\begin{align*}
    (\koop_f^n g)(x) & = \sum_{1\leq i,j\leq n} \mathbb{K}_{M,j,i}^{-1} (\Omega_n) g(\xi_i)  \Ker_{M,\xi_j}(f(x)) 
\end{align*}
for all $x \in M$ and $g \in H_M$. In the above identity, the term $\mathbb{K}_M^{-1}(\Omega_n)$ represents the inverse of the Grammian matrix $\mathbb{K}_M(\Omega_n):=[\Ker_M(\xi_m,\xi_n)]$ associated with the finite sample set $\Omega_n$. From the above expression, we note that this approximation of the Koopman operator can be computed only when the function $f$ is explicitly known. 

For the data-driven approximation, we use the second approximation of the Koopman operator $\dkoop_f^n$. In this paper, when  constructing the operator $\dkoop_f^n$, we assume that (i) the samples $\Xi$ are dense in the manifold $M$, and (ii) the pullback space $f^*(H_M)$ is a subset of the RKH space $H_M$. Note that the projection operator $P_{\Omega_n}: H_M \to H_M$. The definition of the approximated Koopman operator $\dkoop_f^n$ makes sense only when the second assumption mentioned above is valid. 
A coordinate representation of the data-dependent approximation is given by 
\begin{equation*}
    \dkoop_{f}^n g:=\sum_{1\leq i,j \leq n} \mathbb{K}_{M,j,i}^{-1} (\Omega_n)  h(\xi_i) \Ker_{M,\xi_j},
\end{equation*}
where
\begin{equation*}
    h(\xi_i):=\sum_{1\leq p,q \leq n} \mathbb{K}_{M,q,p}^{-1} (\Omega_n)  g(\xi_p) \Ker_{M,\xi_q}(f(\xi_i)).
\end{equation*}

\noindent If the function $g$ is defined as $g:=\sum_{1\leq j \leq n}c_j \Ker_{M,\xi_j}$, the explicit representation of $\dkoop_f^n$ is given as
\begin{equation}
\dkoop_{f}^n g:=\sum_{i,j,m} c_j  
\Ker_{M,\xi_j}(y_i)\mathbb{K}_{M,m,i}^{-1}(\Omega_n) \Ker_{M,\xi_m}. \label{eq:coord_Unf}
\end{equation}

\begin{thm}
\label{th:many_M}
Suppose that $M$ is a $d$-dimensional, connected, compact, Riemannian manifold without boundary, let $\Ker_M:M\times M\rightarrow \RR$ be a positive definite kernel that induces a native space $H_M$, and suppose that $H_M$ is equivalent to the Sobolev space $W^{t,2}(M)$ for some $t\in \RR$  that satisfies  $d/2<s \leq \lceil t \rceil - 1$ for a given $s\in \NN$. Then there are constants $C_M,h_M>0$ such that for all $\Omega_n\subset \Omega$ that satisfy $h_{\Omega_n,\Omega}\leq h_M$, we have 
\begin{align*}
    \|\koop_f g - \koop_f^n g\|_{f^*(W^{s,2}(M))} \leq C_M h_{\Omega_n,M}^{t-s}\|g\|_{W^{t,2}(M)}
\end{align*}
for $g\in W^{t,2}(M)$. 
\end{thm}
\begin{proof}
Since $s>d/2$, the Sobolev embedding theorem implies that  $W^{s,2}(M)$ is a RKH space, and therefore the pullback space  $f^*(W^{s,2}(M))$ is a well-defined RKH space. By the definition of the pullback space we have 
$$
\|\koop_f g - \koop_f^n g\|_{f^*(W^{s,2}(M))} \leq \|\koop_f\| \|(I-P_{\Omega_n})g\|_{W^{s,2}(M)}.
$$
By definition of the norm of the pullback space, we have
\begin{align*}
\|U_f g\|_{f^*(W^{s,2}(M))}
&=\min\{ \|h\|_{W^{s,2}(M)} \ | \\ & \ \ \ \ \ \ \ \ \  \ \ g\circ f =h\circ f, \ h\in W^{s,2}(M) \} \\
&\leq \|g\|_{W^{s,2}(M)}, 
\end{align*}
which implies that $\|\koop_f\| \leq 1$. Additionally, we know that $((I-P_{\Omega_n})g)|_{\Omega_n}=0$ on $\Omega_n$ since the projection is identical to the interpolant over $\Omega_n$. By the many zeros Theorem \ref{th:many_zeros_M}, we conclude that 
$$
\|\koop_f g - \koop_f^n g\|_{f^*(W^{s,2}(M))} \leq C_Mh_{\Omega_n,M}^{t-s}\|g\|_{W^{t,2}(M)}.
$$
\end{proof}
The bound above is stated in terms of the norm on the pullback space $f^*(H_M)$, which may seem rather abstract.  The following corollary illustrates that such a bound naturally leads to a more intuitive pointwise bound. 
\begin{cor}
\label{thm:koopratePW}
Suppose that the hypotheses of Theorem \ref{th:many_M} hold. There are constants $C_M,h_M>0$ such that for all $\Omega_n\subset \Omega$ that satisfy $h_{\Omega_n,\Omega}\leq h_M$, we have 
the pointwise bound 
\begin{align*}
    |(\koop_f &g)(x) - (\koop_f^ng)(x)| \leq C_M h_{\Omega_n,M}^{t-s} \|g\|_{W^{t,2}(M)}
\end{align*}
for all $g\in W^{t,2}(M)$ and $x\in M$. 
\end{cor}
\begin{proof}
First, we note that
\begin{align*}
    |(\koop_f &g)(x) - (\koop_f^ng)(x)| 
    = |g(f(x)) - (P_{\Omega_n}g)(f(x))| \\
    & \leq \sup_{\eta \in M} |g(\eta) - (P_{\Omega_n}g) (\eta)| = \| (I - P_{\Omega_n}) g \|_{C(M)}.
\end{align*}
Since $s > d/2$, by the Sobolev embedding theorem, there exists a constant $K$ such that
$
    \| (I - P_{\Omega_n}) g \|_{C(M)} \leq K \| (I - P_{\Omega_n}) g \|_{W^{s,2}(M)}.
$
From Theorem \ref{th:many_zeros_M}, we can conclude that 
\begin{align*}
    |(\koop_f &g)(x) - (\koop_f^ng)(x)| \leq C_M h_{\Omega_n,M}^{t-s} \|g\|_{W^{t,2}(M)}.
\end{align*}
\end{proof}

The final, principal result of this paper uses the above to derive pointwise bounds for data-driven approximations of the Koopman operator. 
\begin{thm}
\label{th:many_M_data}
Suppose that the hypotheses of Theorem \ref{th:many_zeros_M} holds.  Furthermore, suppose that the mapping $f:M\rightarrow M$ is such that the pullback space $f^*(H_M)\subseteq H_M$. Then there are constants $C_M,h_M>0$ such that for all $\Omega_n\subset \Omega$ that satisfy $h_{\Omega_n,\Omega}\leq h_M$, we have 
the pointwise bound 
\begin{align*}
        |(\koop_f &g)(x) - (\dkoop_f^ng)(x)| \\
    &\leq C_M h_{\Omega_n,M}^{t-s}\left ( \|g\|_{W^{t,2}(M)} + \|\koop_f  g\|_{W^{t,2}(M)} \right )
\end{align*}
for $g\in W^{t,2}(M)$ and $x\in M$. 
\end{thm}
\begin{proof}
Since $f^*(H_M)\subseteq H_M$, we know that $P_{\Omega_n}g\in H_M$ and $P_{\Omega_n}g\circ f\in f^*(H_M)\subset H_M$. By definition, we have
\begin{align*}
    |(\koop_f g)(x) - (\dkoop_f^n &g)(x)| \\
    &:= \left |(g\circ f)(x)- (P_{\Omega_n}\left ((P_{\Omega_n} g)\circ f)\right )(x)\right | \\
    &\leq |(g\circ f)(x)-(\Pro{n}g)\circ f)(x)  | \\
    & \qquad + |((I-\Pro{n})(\Pro{n}g\circ f))(x) |.
\end{align*}
Under the hypotheses of this theorem, we have 
$
H_M\approx W^{t,2}(M)\hookrightarrow W^{s,2}(M) \hookrightarrow  C(M). 
$
This implies that there are positive constants $K_1$ and $K_2$ such that 
\begin{align*}
 |(&\koop_f g)(x) - (\dkoop_f^n g)(x)| \\
  &\leq \sup_{\eta\in M} |(((g-\Pro{n}g)\circ f)(\eta)|  \\
  & \qquad + \sup_{\eta\in M} |((I-\Pro{n})(\Pro{n}g\circ f))(\eta) | \\
  &\leq \| (I-\Pro{n})g\|_{C(M)} + \| (I-\Pro{n})(\Pro{n}g\circ f) \|_{C(M)} \\
  &\leq K_1 \| (I-\Pro{n})g\|_{W^{s,2}(M)} 
  \\
  & \qquad + K_2 \| (I-\Pro{n})(\Pro{n}g\circ f) \|_{W^{s,2}(M)}
\end{align*}
Now we apply the many zeros Theorem \ref{th:many_M} to each of the right hand side terms above. We know that 
$
\left (I-\Pro{n})g \right)|_{\Omega_n}=0,
\left (I-\Pro{n})(P_{\Omega_n}g\circ f) \right)|_{\Omega_n}=0 
$
since the projection operator $\Pro{n}$ is identical to the interpolation operator on $\Omega_n\subset M$. By the many zeros theorem on the manifold $M$, we get
\begin{align*}
|(&\koop_f g)(x)- (\dkoop_f^n g)(x)|\\
& \leq C_1 h_{\Omega_n,M}^{t-s}\|g\|_{W^{t,2}(M)}  + C_2 h^{t-s}_{\Omega_n,M}\|\Pro{n}g\circ f\|_{W^{t,2}(M)} \\
&\leq \max \{C_1,C_2\} h^{t-s}_{\Omega_n,M} \left ( \|g\|_{W^{t,2}(M)} + \|\koop_f  g\|_{W^{t,2}(M)} \right ).
\end{align*}
\end{proof}

\section{NUMERICAL EXAMPLE}
\label{sec:examples}
\begin{figure}
    \centering
    \includegraphics[width=.4\textwidth]{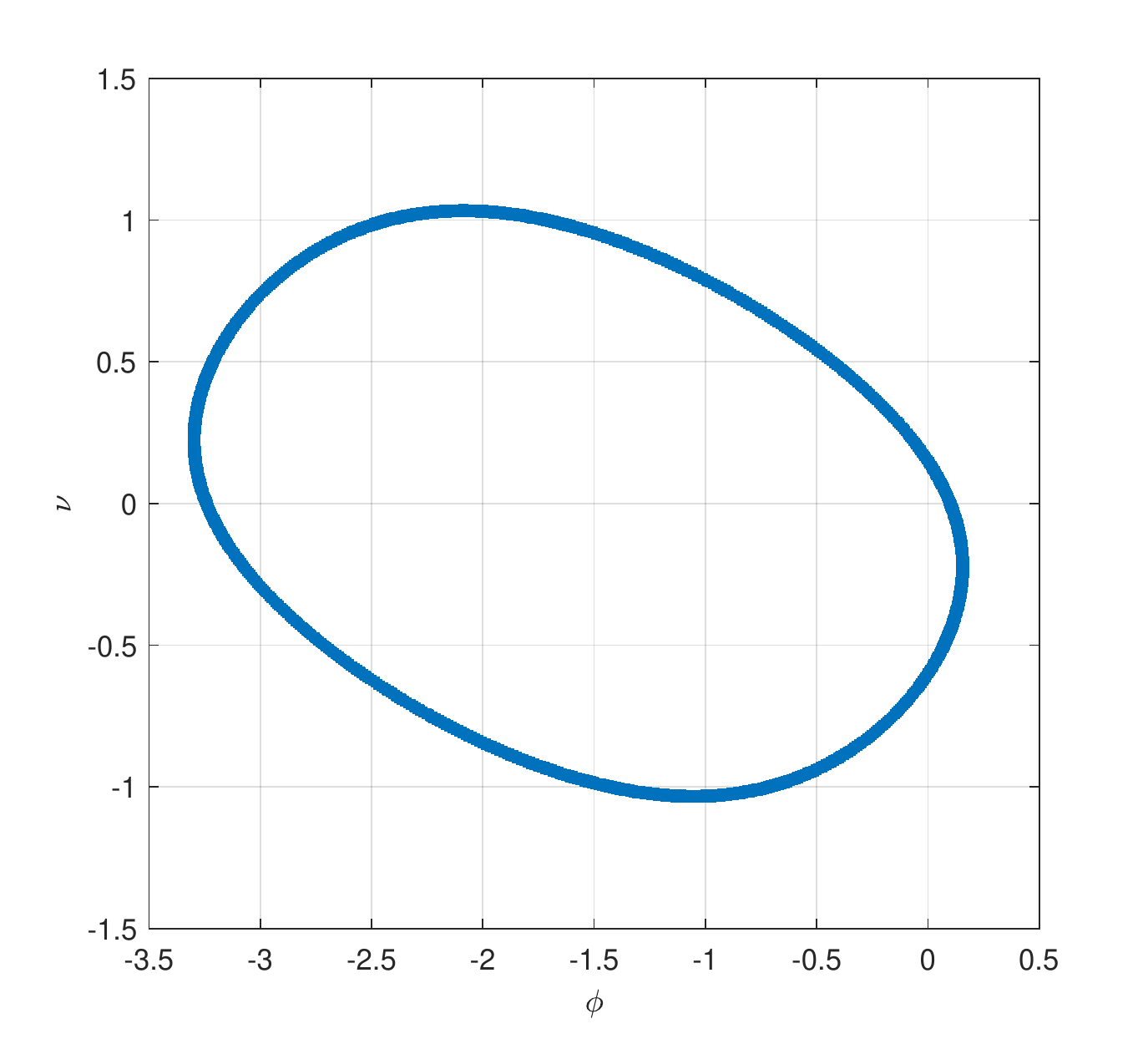}
    \caption{The Discrete Trajectory}
    \label{fig:figPhasePlot}
\end{figure}

In this section, we study the application of the derived bounds on rates of convergence to the classical model of a bouncing ball on a vibrating surface. 
The difference equation that defines  the state trajectory is given by
\begin{align*}
\phi_{j+1} &= \phi_j + \nu_j, \\
\nu_{j+1} &= \alpha \nu_j - \gamma \cos(\phi_j + \nu_j),
\end{align*}
where $\phi$ and $\nu$ are the nondimensional impact time and the velocity after impact, respectively. The constants $\alpha$ and $\gamma$  in the above equation represent the dissipation coefficient and force amplitude, respectively. We refer the reader to \cite{Holmes1982} for a more detailed discussion of this dynamical system. Figure \ref{fig:figPhasePlot} shows the state trajectory generated by this system when $\alpha = 1$ and $\gamma = 0.45$, $[\phi_0,\nu_0]^T = [0.1,0]^T$ after $1024$ iterations. 
The function $f:\mathbb{R}^2 \to \mathbb{R}^2$ in this case is given by $f([\phi_j,\nu_j]) \to [\phi_{j+1},\nu_{j+1}]^T$. For purposes of illustration, we choose the observable  function $g:\mathbb{R}^2 \to \mathbb{R}$ defined as
$
g([\phi,\nu]^T) = \phi + \nu.
$

\subsection{Challenges to Intrinsic Approximations}
This example has been selected in part to emphasize some of the inherent difficulties when seeking to generate bounds on rates of approximation of Koopman operators by intrinsic methods. In view of Figure \ref{fig:figPhasePlot}, it seems reasonable to believe that the samples $\Xi$ are dense in a smooth, one-dimensional, regularly embedded submanifold $M$ of $X:=\RR^2$. Even though this example is exceptionally straightforward, where the mapping $f:X\rightarrow X$ and the observable $g:X\rightarrow \RR$ are known in closed form, it remains  difficult to employ the bounds in Theorems \ref{th:many_M} through \ref{th:many_M_data} in an intrinsic approximation over $M$. To employ the results of these theorems, we would first need to define some Riemannian metric on $M$. Theoretically this is always possible if $M$ is a smooth manifold.   Subsequently we must  define an appropriate kernel $\mathfrak{K}_M:M\times M \rightarrow \RR$ whose native space is equivalent to a Sobolev space. {\em In principle}, this too can be accomplished.  For example, we can solve for the fundamental solution of a sufficiently high order of the Laplace-Beltrami operator over $M$, which could then be taken as the kernel of $H_M$. By definition we would obtain $H_M\approx W^{t,2}(M)$ for some $t>0$, see \cite{HNW} and the references therein for a general  discussion.  With such a definition of the kernel $\Ker_M:M\times M \rightarrow \RR$, the  results of the theorems in this paper would then apply.  However, even in this remarkably simple example, it is no simple feat to solve for the fundamental solution over $M$. Pragmatically speaking, we do not have a closed form expression for an atlas for $M$, and consequently we cannot solve the coordinate representations of the equation defining the fundamental solution. It would seem that constructing a kernel that is intrinsic to $M$ would be prohibitively difficult in this case.  

We should note of course, that not all examples pose such problems for intrinsic approximations. If  the samples $\Xi$ of the semiflow  are dense in some well-known manifold for which the solution of the Laplace-Beltrami operator equation is known, then the approximations and theorems in this paper are directly applicable. For example, there are a number of classical examples of continuous semiflows whose orbits are dense in the torus. In such a case $X=\Omega=M$ is the torus. It is always assumed that the state space $X$ is known {\em a priori}. The  kernels over the torus  are known in closed form, and these expressions can be used directly in Koopman operator approximations.

\subsection{Explicit Approximations}
Fortunately, the theorems in this paper are easily applied for certain types of extrinsic approximations. We briefly outline the process. 
The Sobolev-Matern kernels $\mathfrak{K}_{X,\nu}:X\times X \rightarrow \RR$  on $X=\RR^2$ are known in closed form, and they induce a native space $H_X$ that is contained in the Sobolev space $W^{\tau,2}(\RR^p)$ for $\tau < 2 \nu - p/2$. Note, the term $\nu$ is a  positive parameter that defines a family of Sobolev-Matern kernels. By the trace theorem, the restriction $\mathfrak{K}_{M,\nu}:=\mathfrak{K}_X|_M$ of the kernel $\mathfrak{K}_{X,\nu}$ induces a native space $H_M$ over the manifold $M$ that is contained in the Sobolev space $W^{t,2}(M)$, where $t<\tau-(p-d)/2$. In our example, a 1-dimensional manifold is contained in $\RR^2$, and hence $p = 2$ and $d = 1$. Note that there is some loss of smoothness in restricting functions in $H_X\approx W^{\tau,2}(X)$ to the regularly embedded submanifold $M$ in that  $H_M\approx W^{t,2}(M)$.   

In this simulation, we use the Sobolev-Matern kernel $\mathfrak{K}_{X,\nu=5/2}$, which has the form $\mathfrak{K}_{X,\nu=5/2}(x,y)= \mathcal{K}(\|x-y\|)$, where
\begin{align*}
    \mathcal{K}(r) := \left( 1 + \frac{\sqrt{5}r}{l} + \frac{5 r^2}{3 l^2} \right) \exp{\left( - \frac{\sqrt{5}r}{l} \right)}.
\end{align*} 
In the above equation, the term $l$ is a positive parameter, and we obtained the numerical results of this paper with $l = 1e-1$. Note that the above kernel is defined over $X = \mathbb{R}^2$ and its RKH space is contained in $W^{\tau,2} (\RR^2)$, where $\tau < 4$.

\begin{figure}
    \centering
    \includegraphics[width=.4\textwidth]{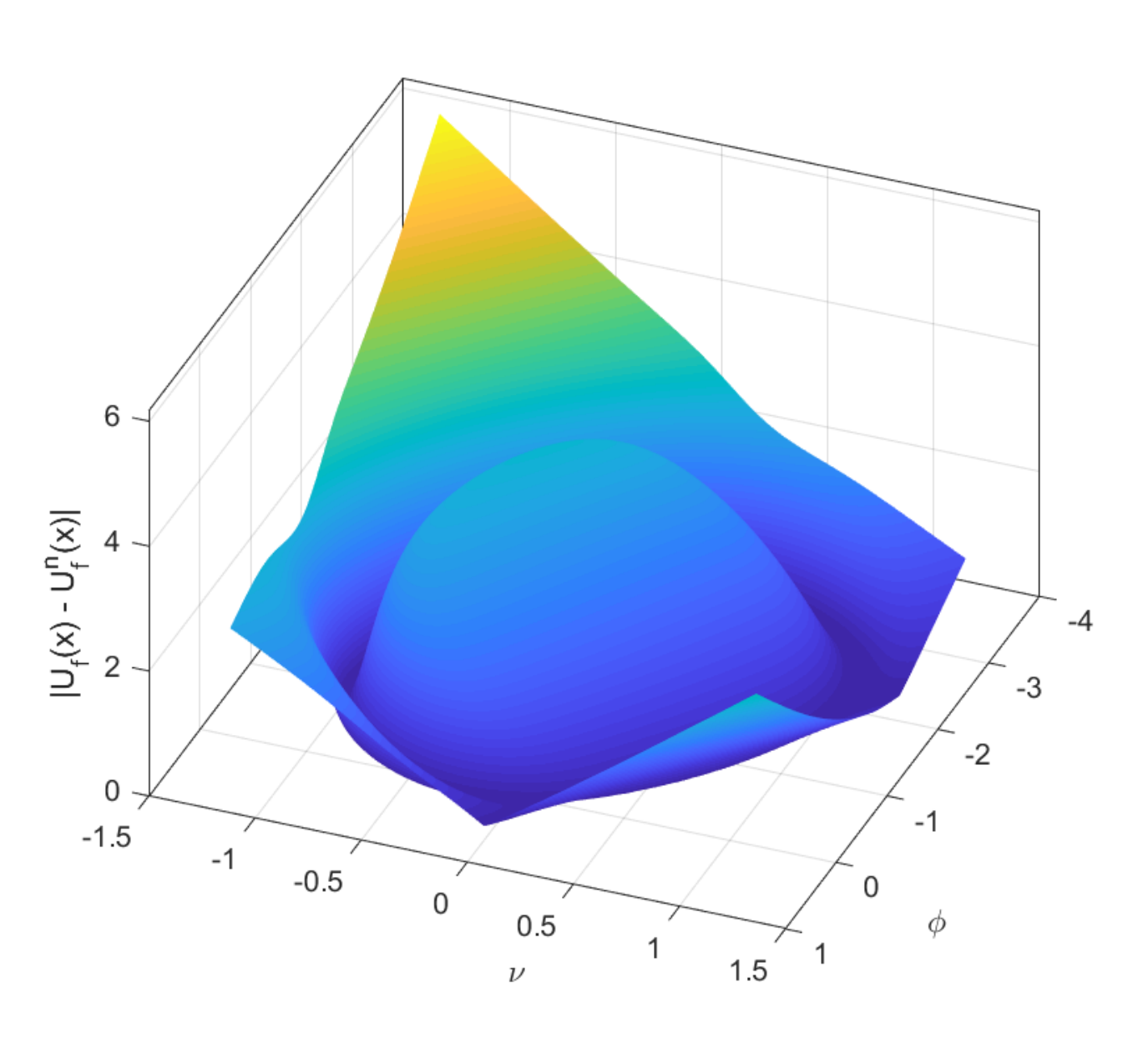}
    \caption{Pointwise Error, $\koop_f^n$ Approximation}
    \label{figErr3D}
\end{figure}

\begin{figure}
    \centering
    \includegraphics[width=.4\textwidth]{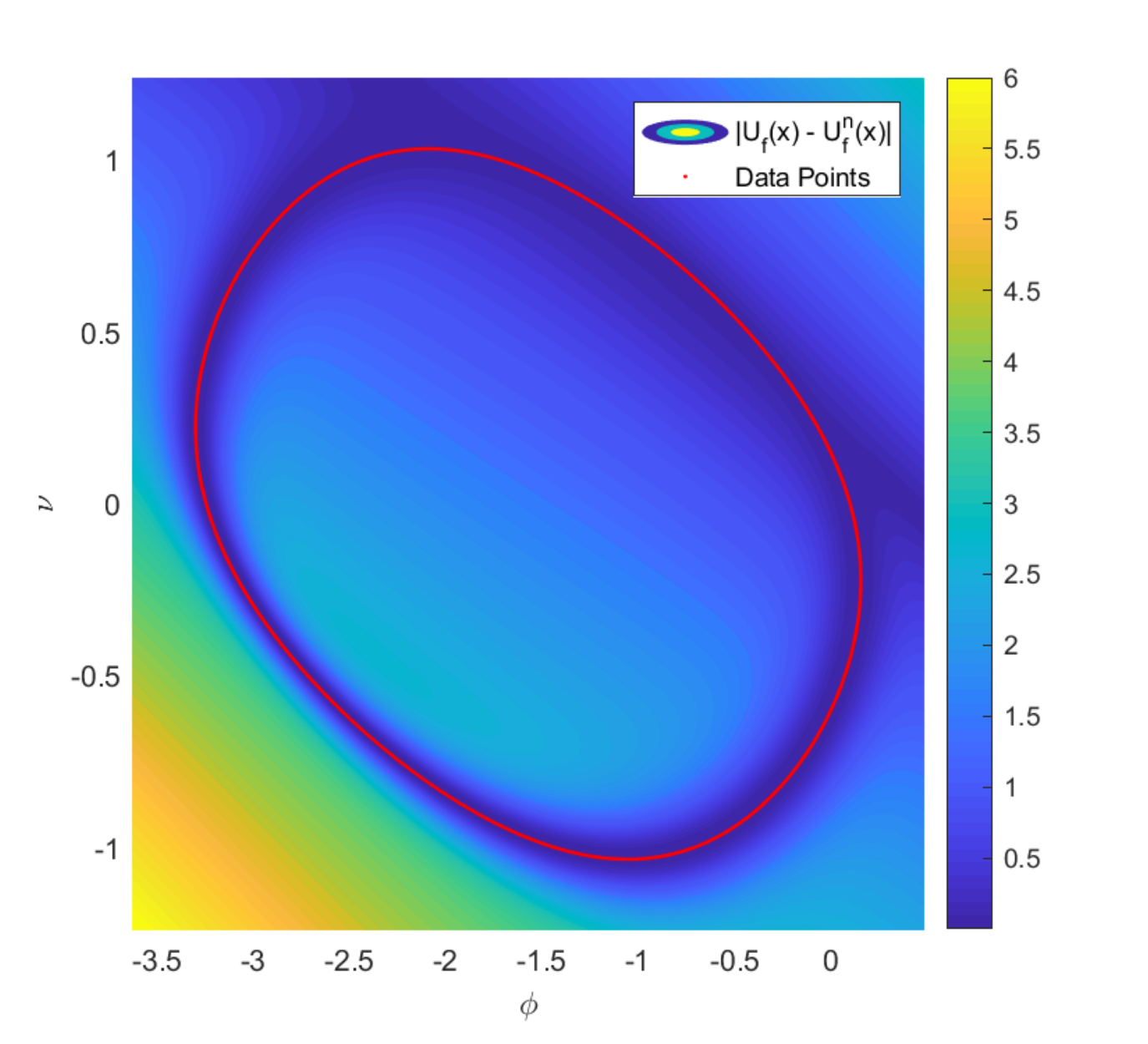}
    \caption{Error Contour,  $\koop_f^n$ Approximation}
    \label{figErrCont}
\end{figure}

The pointwise error $|U_f(x) - U_f^n(x)|$ in $\mathbb{R}^2$ for $n = 768$ is shown in Figure \ref{figErr3D}. The kernels for this simulation were centered at the first $768$ data points generated by the dynamical system. Figure \ref{figErrCont} shows the error contour. As expected, the error is minimized over the manifold. The error plots for the data driven approximation of the Koopman operator is similar.

Figure \ref{figKoopInfBnd} shows how the $C^\infty$-norm error $\|U_f - U_f^n\|_{\infty}$ varies as the fill distance $h$ is decreased. Since the manifold M is not explicitly defined, we use the Euclidean metric to calculate the fill distance $h$. It is straightforward to show that the Euclidean metric is equivalent to the intrinsic metric of $M$ since $M$ is a regularly embedded manifold. Since we are plotting the variables on a log scale, the slope of the error lines should be less than or equal to $t-s < 2.5$. Note, the constant $t<3.5$ is defined by the choice of the kernel and the constant $s$ satisfies $d/2 = 0.5 < s = 1 \leq \lceil 3.5 \rceil - 1$. Figure \ref{figDataKoopInfBnd} shows the equivalent plot for the data-driven Koopman approximation. From these plots, it is clear that the error decays at a rate higher than the worst-case theoretical bound of $t-s < 2.5$.



\begin{figure}
    \centering
    \includegraphics[width=.4\textwidth]{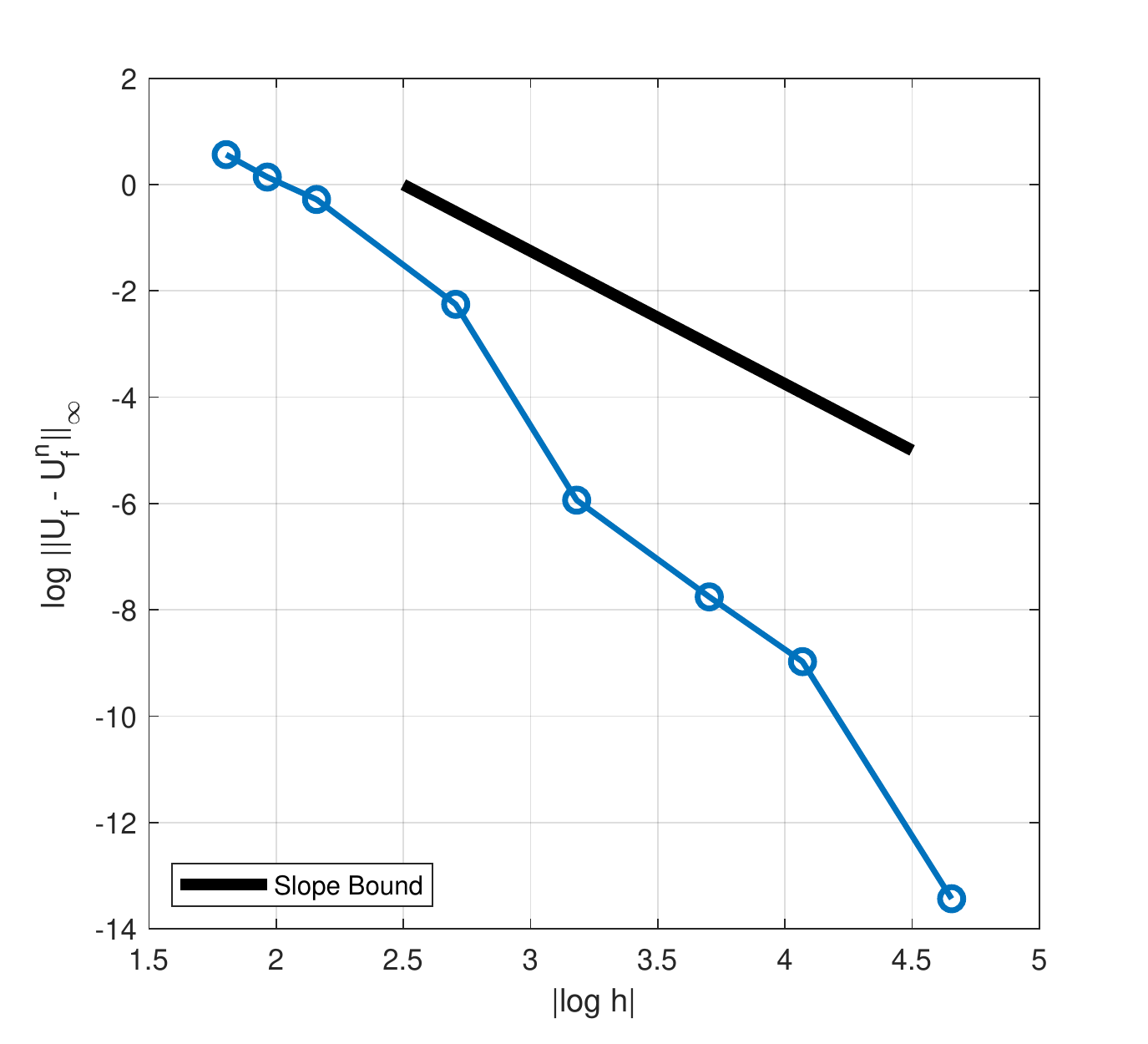}
    \caption{$C^\infty$-norm Error, $\koop_f^n$ Approximation}
    \label{figKoopInfBnd}
\end{figure}

\begin{figure}
    \centering
    \includegraphics[width=.4\textwidth]{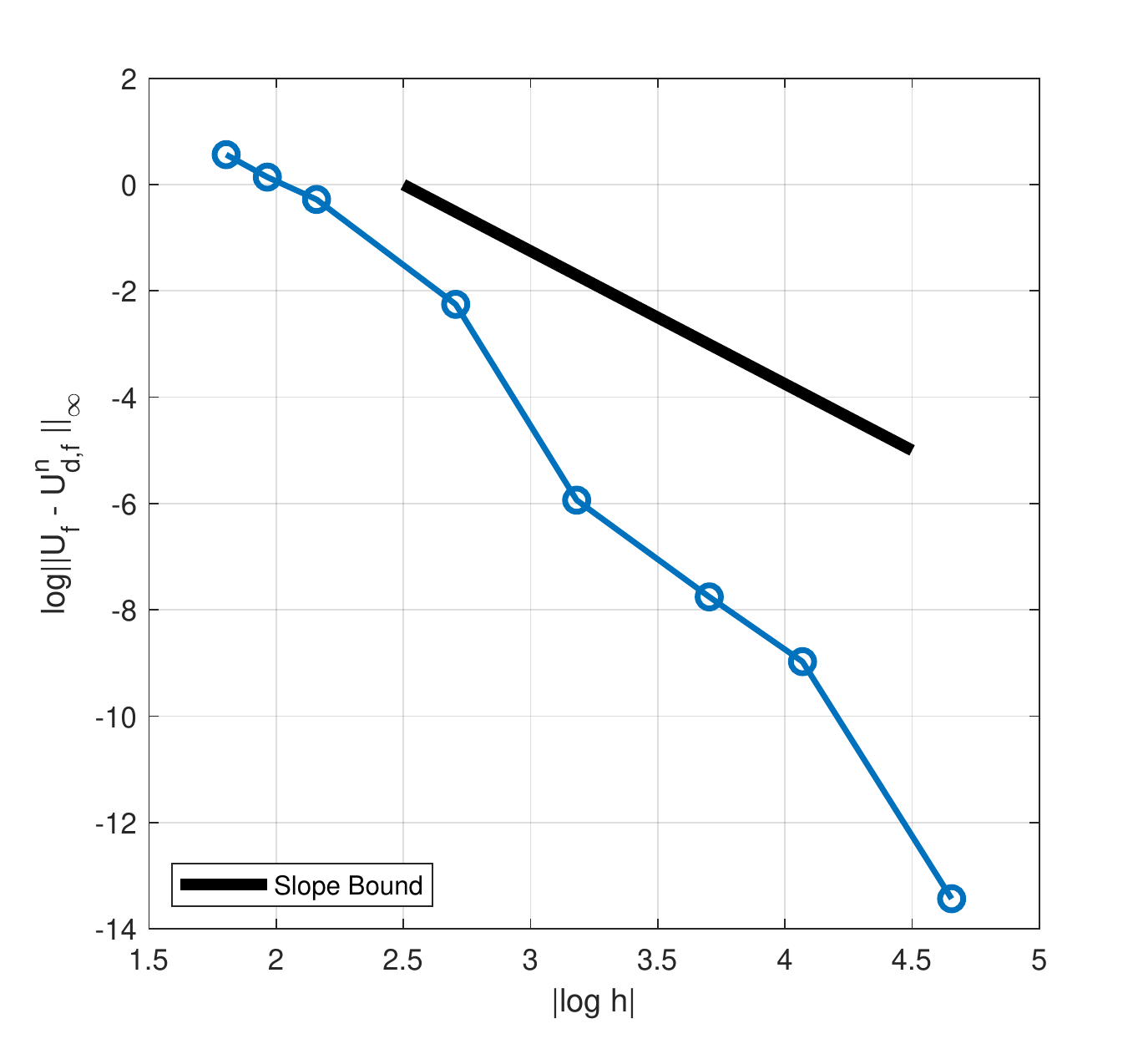}
    \caption{$C^\infty$-norm Error, $\dkoop_f^n$ Approximation ($\dkoop_f^n \equiv U_{d,f}^n$)}
    \label{figDataKoopInfBnd}
\end{figure}

\section{CONCLUSIONS AND FUTURE WORK}

This paper has derived explicit error bounds on projection-based and data-driven approximations $\koop_f^n$ and $\dkoop_f^n$, respectively, of the Koopman operator $\koop_f$ when the samples $\Xi$ are dense in a smooth Riemannian manifold $M$ and the number of samples $m$ is equal to the dimension of the space of approximants $n$. Numerical studies illustrate the qualitative nature of the convergence rates: convergence is achieved over the manifold $M$ and the rate of convergence is bounded above by the expressions derived in the theorems that depend on the fill distance.

While the numerical results do provide some validation of the theoretical results, they are preliminary and illustrate worst-case performance of the Koopman operator approximations.  Since functions $f$ and  $g$ are quite smooth, the rates of convergence of $\koop_f^ng$ and $\dkoop_f^ng$ to $\koop_f g$  are much faster than the worst-case bounds.  Future numerical studies should investigate how convergence rates vary with more irregular or nonsmooth functions $f$ and $g$.

\bibliographystyle{IEEEtran}
\bibliography{approx1}

\end{document}